\documentclass[10pt]{article}
\usepackage{mathrsfs}
\usepackage{amssymb}
\usepackage{color}
\usepackage{amsthm}
\usepackage{amsmath}
\usepackage{graphicx}
\usepackage[square, comma, sort&compress, numbers]{natbib}

 \newtheorem{thm}{Theorem}[section]

 \newtheorem{prop}[thm]{Proposition}
 \theoremstyle{definition}

 \numberwithin{equation}{section}

\newcommand{\Rmnum}[1]{\expandafter\@slowromancap\romannumeral #1@}
\makeatother
\newtheorem{theorem}{Theorem}[section]
\newtheorem{lemma}[theorem]{Lemma}

\newtheorem{proposition}[theorem]{Proposition}

\theoremstyle{definition}

\theoremstyle{remark}

\newcommand \bel {\be\label}
\newcommand \del \partial

\newcommand \be {\begin{equation}}
\newcommand \ee {\end{equation}}

 \newcommand \RR			{\mathbb R}

\newcommand \underm {\underline m}

\newcommand \bei  {\begin{itemize}}
\newcommand \eei {\end{itemize}}

\newcommand \Et {\widetilde E}

\begin{document}
\title{Boundedness of the total energy of relativistic membranes
 evolving in a curved spacetime}
\author{Philippe G. LeFloch\footnotemark[1] \hskip.15cm and
 Changhua Wei\footnotemark[2]} 

\footnotetext[1]{Laboratoire Jacques-Louis Lions \& Centre National de la Recherche Scientifique, Universit\'e Pierre et Marie Curie (Paris 6), 4 Place Jussieu, 75252 Paris, France. Email: contact@philippelefloch.org}

\footnotetext[2]{Corresponding author: 
Department of Mathematics, Zhejiang Sci-Tech University, Hangzhou 310018, P.R. China.  
Email: changhuawei1986@gmail.com.
}

\date{September 2016}

\maketitle
\allowdisplaybreaks[4]

\begin{abstract}
We establish a global existence theory for the equation governing the evolution of a relativistic membrane in a (possibly curved) Lorentzian manifold, when the spacetime metric is a perturbation of the Minkowski metric. Relying on the Hyperboloidal Foliation Method introduced by LeFloch and Ma in 2014, we revisit a theorem established earlier by Lindblad (who treated membranes in the flat Minkowski spacetime) and we provide a simpler proof of existence, which is also valid in a curved spacetime and, most importantly, leads to the important property that the total energy of the membrane is globally bounded in time. 
\end{abstract}




\section{Introduction}
\label{Section1}

\subsection*{Objective of this paper}

We are interested in the following $n$-dimensional relativistic membrane equation
\bel{1.1}
D_\mu \Big(\frac{D^\mu\phi}{\sqrt{1+ D_{\nu}\phi D^{\nu}\phi}}\Big) = 0,
\ee
posed in a curved spacetime with Lorentzian metric $g=(g_{\mu\nu})$ and  covariant derivative operator $D_\mu $.
The unknown function $\phi$ defined in Minkowski spacetime $\RR^{n+1}$
parametrizes the timelike hypersurface of interest. We use Einstein convention on repeated indices and
all Greek indices $\mu, \nu, \ldots$ take their values in $\{0,1,\cdots,n\}$.

Relativistic membrane equations are the hyperbolic counterparts to the minimal surfaces/subma\-nifolds, which has a long history originating from the famous Plateau problem and been studied pretty well in the last decades; see for instance Osserman \cite{Osserman}. The theory of relativistic membrane also plays a very important role in elementary particle physics, especially in dealing with (one-dimensional) relativistic strings. For the study of relativistic strings in classical dynamics and via the first quantization method, we refer to the textbook of Barbashov and Nesterenko \cite{B-N}. The theory of relativistic membranes attempts to develop a quantum theory of higher-dimensional, relativistic, extended subjects, and describes the dynamics of many physical systems under appropriate phenomenological assumption, such as Nambu action. 

In the present paper, we study solutions to the initial value problem associated with \eqref{1.1} which was first solved by Lindblad \cite{Lindblad} at least for membranes in Minkowski spacetime. Our main objective is to improve the result therein and to establish that the natural energy of the membrane is {\sl globally bounded in time}, while \cite{Lindblad} could only obtain a bound for lower-order derivatives of the solution. Our proof uses LeFloch and Ma's Hyperboloidal Foliation Method \cite{PLF-YM-book}--\cite{PLF-YM-2}, which allows one to treat coupled systems of wave and Klein-Gordon equations. This method was built upon earlier work by Klainerman \cite{Klainerman} and Hormander \cite{Hormander} on the quasilinear Klein-Gordon equation. The use of the hyperboloidal foliation of Minkowski spacetime to study coupled systems of wave and Klein-Gordon equation was investigated in \cite{PLF-YM-book} and several classes of such systems were then treated by this method, including the Einstein equations of general relativity  \cite{PLF-YM-book,PLF-YM-1,PLF-YM-2}. Estimates in the hyperboloidal foliation are often more precise and it was observed in  \cite{PLF-YM-book} that, for certain classes of equations, the energy of the solution is bounded globally in time. See LeFloch-Ma's theorem in \cite[Chap.~6]{PLF-YM-book}. 
Our aim in the present paper is to extend this observation to the membrane equation \eqref{1.1}. Our proof will, in addition, rely on a ``double null'' property of the membrane equation which was first observed by Lindblad \cite{Lindblad}.

\subsection*{Statement of the theorem}

By the property of finite speed of propagation of the relativistic membranes, we at first specify the domain of the nonvanishing functions, which is the interior of the future light cone from the point $(1,0,0)$
$$
K:=\{(t,x)\,|\,r<t-1\},
$$
and the following domain limited by two hyperboloids (with $s_{0}<s_{1}$)
$$
K_{[s_{0},s_{1}]}:=\{(t,x)\,|\,s^{2}_{0}\leq t^{2}-r^2\leq s_{1}^2;\,r<t-1\}.
$$
In addition, in this paper we are also interested in covering a broad class of metrics
and, in coordinates $(t,x) = (x^0,x^1,\cdots,x^n)$ with $g=g_{\mu\nu} dx^\mu dx^{\nu}$
and choosing the signature $(-,+,\cdots,+)$, we assume that the metric is a perturbation of the Minkowski metric
represented by $diag(-1,1,\cdots,1)$, more precisely:
\be
\aligned
|g^{\mu\nu}-m^{\mu\nu}| &\leq \delta\Big(\frac{s}{t}\Big)^2s^{-1-\gamma},
\qquad
|\del^{I} g_{\mu\nu}| \leq  \delta\Big(\frac{s}{t}\Big)^2s^{-|I|-\gamma} 
\endaligned
\ee
for some small constants $\delta, \gamma>0$. These decay conditions are compatible with the properties established in for solutions to the Einstein equations.
 $\del^{I}=\del_{0}^{I_{0}}\del_{1}^{I_{1}}\cdots\del_{n}^{I_{n}}$ with $I=(I_{0},I_{1},\cdots,I_{n})$ and $|I|=\sum_{i=1}^{n}|I_{i}|$, $s=\sqrt{t^2-r^2}=\sqrt{t^{2}-|x|^2}$.

Expanding the equation \eqref{1.1}, we find 
\bel{1.5}
(1+  \phi^{\nu}\phi_{\nu})\Box_g \phi-D^\mu\phi D_{\nu}\phi D_\mu D^{\nu}\phi=0,
\ee
where
$$
\Box_g =D_\mu D^\mu=g^{\mu\nu}\del^{2}_{\mu\nu}-\Gamma^\mu\del_\mu =\frac{1}{\sqrt{|det g|}}\del_\mu
\Big(\sqrt{|det g|}g^{\mu\nu}\del_{\nu}\Big),
$$
and
$$
\Gamma^{\gamma}_{\alpha\beta}=\frac{1}{2}g^{\gamma\delta}(\del_\alpha g_{\delta\beta}
+\del_\beta g_{\alpha\delta}-\del_{\delta}g_{\alpha\beta}),\qquad \Gamma^\mu=g^{\alpha\beta}\Gamma_{\alpha\beta}^\mu.
$$
Expanding \eqref{1.5}, we have
\begin{eqnarray}\label{1.6}
&&g^{\mu\nu}\del^{2}_{\mu\nu}\phi-\Gamma^\mu\del_\mu \phi
+\Big(g^{\mu\nu}g^{\alpha\beta}\del_\alpha \phi\del_\beta \phi\del^{2}_{\mu\nu}\phi
-g^{\alpha\beta}\Gamma^\mu\del_\alpha \phi\del_\beta \phi\del_\mu \phi\nonumber\\
&-&g^{\mu\alpha}g^{\nu\beta}\del_\alpha \phi\del_\beta \phi\del^{2}_{\mu\nu}\phi
+g^{\mu\alpha}g^{\nu\beta}\Gamma^{\gamma}_{\mu\nu}\del_\alpha \phi\del_\beta \phi\del_{\gamma}\phi\Big) = 0.
\end{eqnarray}

For instance, when the background spacetime is exactly Minkowski spacetime, then $\Gamma^\mu=\Gamma^{\gamma}_{\mu\nu}=0$, and the nonlinear terms take the following form
$$
(-(\del_{t}\phi)^{2}+|\nabla\phi|^{2})(-\del_{t}^{2}\phi+\Delta\phi)
-(\del_{t}\phi)^{2}\del^{2}_{t}\phi+2\del_{t}\phi
\nabla\phi\cdot\del_{t}\nabla\phi
-\sum_{i,j=1}^{n}\del_{i}\phi\del_{j}\phi\del^{2}_{ij}\phi.
$$
Clearly, the above expression satisfies the classical null condition (introduced in \cite{Klainerman}). Furthermore, the nonlinearities are cubic and, based on these facts, Lindblad  \cite{Lindblad} relied on Klainerman's classical vector field method and established the global stability of the trivial solution (for any $n\geq2$).

We prescribe an initial data for \eqref{1.6} on the spacelike surface $H_{s} $ with constant $s$-slice, i.e. 
\bel{1.7}
 \phi(t,x)|_{H_{s_{0}}}=\epsilon f(x),\qquad \del_{t}\phi(t,x)|_{H_{s_{0}}}=\epsilon g(x),
\ee
in which $s_{0}=\sqrt{t^{2}-r^{2}}=B+1>1$, 
and 
$f=f(x)$ and $g=g(x)$ belong to $C_{0}^\infty(\mathbb R^{n})$ --the space of smooth functions with compact support. 
Before we can state our main result, we need to introduce some notation.

As stated before, Greek indices range from 0 to n and Latin indices range from 1 to n. We denote by $Z$ the family of ``admissible vector fields'' consisting by definition of all vectors
$$
Z_\alpha :=\del_\alpha , \qquad Z_{2+a}:=L_{a}=x^{a}\del_{t}+t\del_{a}.
$$
We also use the semi-hyperboloidal frame defined by
\bel{1.8}
\underline{\del}_{0}:=\del_{t},\quad \underline{\del}_{a}:=t^{-1}L_{a}=\frac{x^{a}}{t}\del_{t}+\del_{a}.
\ee
which satisfies the following commutation relations:
\bel{C}
[\underline{\del}_{t},\underline{\del}_{a}]=-\frac{x^{a}}{t^2}\underline{\del}_{t},
\quad [\underline{\del}_{a},\underline{\del}_{b}]=0.
\ee
Our energy functionals are defined by
\bel{1.9}
E_{0}(s,\phi) = \int_{H_{s}}\Big(\sum_{a}(\underline{\del}_{a}\phi)^{2}+((s/t)\del_{t}\phi)^{2}\Big) \, dx,
\ee
\bel{1.10}
E_{m}(s,\phi):=\sum_{|I|\leq m}E_{0}(s,Z^{I}\phi).
\ee
We are now in a position to state the main result of this paper.

\begin{theorem}[Boundedness of the energy for relativistic membranes]
\label{Theorem1.4}
There exists a sufficiently small parameter $\epsilon_{0}>0$ such that the Cauchy problem \eqref{1.1} and \eqref{1.7} admits a unique global in time solution provided $\epsilon \in (0, \epsilon_{0})$. Furthermore, provided $m\geq6$, one has the uniform energy bound
\be
E_{m}(s,\phi)\lesssim \epsilon^2.
\ee
\end{theorem}

The following observations are in order: 

\bei

\item Theorem \ref{Theorem1.4} establish the energy \eqref{1.10} remains uniformly bounded. The Hyperboloidal Foliation Method therefore leads us to a better result, in comparison with the classical vector field method, which leads to the possibility that the highest order energy would be slowly increasing with time.

\item Since
 for the linear and nonlinear wave equations, we know tha
 the decay rate of the $L^\infty$-norm of the solutions depends on the spatial dimension and is slower in diension $n=2$, we will present the details of the proof below in the case $n=2$, while it is straighforward to check that all of our arguments below hold true fin any dimension $n\geq3$.

\eei

\subsection*{Background material and outline of this paper}

We present some further background on the study of relativistic membrane equations.
The relativistic membrane equation can be derived from two aspects. On one hand, geometrically, it is used to characterize the timelike extremal surface embedded in a bulk. On the other hand, by the potential theory, it is used to describe the motion of a perfect fluid named relativistic Chaplygin gas, when this fluid is spacetime irrotational and isentropic. For the details, one can refer to \cite{Hoppe4,Matt}. Due to its importance in geometry and physics, the study of the relativistic membrane has attracted a lot of attention in the last decade. When the background manifold is the Minkowski spacetime, Lindblad \cite{Lindblad} proved the global existence of smooth solutions of relativistic membrane equation when the initial data is sufficiently small and has compact support for $n\geq 2$ by the classical vector field method, however, the highest order energy he got is polynomially increasing with respect to time. Later on, utilizing the similar idea of Lindblad \cite{Lindblad}, Allen, Andersson and Isenberg \cite{Allen} generalized Lindblad's result to timelike extremal submanifold with codimension larger than 1. By the method of geometric analysis, Brendle \cite{Brendle} showed the global existence of smooth solution to the relativistic membrane equation under the perturbation of some timelike hyperplane. His method reduced the estimates on the orders of the derivatives, while it only works for the spatial dimension $n> 3$. By investigating the algebraic and geometric properties of relativistic membrane, Bordeman and Hoppe \cite{Hoppe2,Hoppe3} obtained the classical solution in the light cone gauge. Moreover, they find the relationship between relativistic membranes and two dimensional fluid dynamics by performing variables transformations . Based on this fact, Kong, Liu and Wang \cite{K-L-W} proved the global stability of two dimensional isentropic Chaplygin gases without vorticity by the potential theory. Lei and Wei in \cite{Lei-Wei} investigated the relationship between relativistic membrane equation and relativistic Chaplygin gases and with the help of the ``null structure'' of relativistic membrane equation in exterior region, they showed the global radial solutions of 3D nonisentropic relativistic Chaplygin gases.

For the study of relativistic membrane in general curved background manifold, only a few results are available in the literature. He and Kong in \cite{He-Kong} studied the properties of relativistic membrane equations in Schwarzschild spacetime, furthermore, they showed the global existence and blowup of smooth radial solutions according to different assumptions on the initial data. Kong and Wei in \cite{Kong-Wei} investigated the lifespan of smooth solution to the relativistic membrane equation embedded in de Sitter spacetime when the initial data is sufficiently small and has compact support. Luo and Wei \cite{LW} proved the global existence of smooth solutions to 2D isentropic Chaplygin gases without vorticity in curved space when the initial data is small and has compact support. Recently, in order to investigate the influence of the background metric to the stability of the large solution of relativistic membrane equations, Wei \cite{Wei} studied a class of time-dependent Lorentzian metric and showed that when the decay rate of the derivative of the given metric is larger than $3/2$, then a class of time-dependent large solution to the relativistic membrane is globally stable, while, therein, the highest-order energy is still growing with respect to time. 

So, as far as membrane equations are concerned, Theorem~\ref{Theorem1.4} is the first result that provides a uniform energy bound at all orders and also provides the first generalization to LeFloch-Ma's theorem in \cite[Chap.~6]{PLF-YM-book}. 
Our strategy of proof can be described as follows: At first, for the linear terms, since we do not know the sign of $\Gamma^\mu\del_\mu \phi$ and $(g^{\mu\nu}-m^{\mu\nu})\del^2_{\mu\nu}\phi$, we treat these terms as nonlinear terms and get the decay rate from the perturbation of the background metric, which is the most important reason for our constraint to the metric. Then, for the nonlinear terms, we split them into two parts, the one satisfies the classical null condition, the other one has more decay rate due to the perturbation of the metric.

We arrange our paper as follows: In Section \ref{Section2}, we give some necessary preliminaries on the semi-hyperboloidal frame and the energy estimate for the hyperboloidal foliation. In Section \ref{Section3}, we analyze the structure of the nonlinear terms in semi-hyperboloidal frame and give some important estimates for all the linear and nonlinear terms. Based on the Sobolev inequality and the definition of the energy, we estimate the $L^\infty$-norms of the derivatives of the unknowns in Section \ref{Section4} and also calculate the decay rate of the unknowns under the assumption that the highest order energy is uniformly bounded. In Section \ref{Section6}, we close the above a priori assumption by the continuity method.


\section{Preliminaries on the hyperboloidal foliation}
\label{Section2}

\subsection*{Notations}

We begin with some material from LeFloch and Ma's monograph \cite{PLF-YM-book}. The transition matrices between the semi-hyperboloidal frame \eqref{1.8} and the usual frame $\del_\alpha $, are given by $\underline{\del}_\alpha =\Phi_\alpha ^{\beta}\del_\beta $ and $\del_\alpha =\Psi^{\beta}_\alpha \underline{\del}_\beta $, where
\bel{2.1}
\Phi:=\left(
\begin{array}{ccc}
1&0&0\\
x^{1}/t&1&0\\
x^{2}/t&0&1
\end{array}\right),
\qquad
\Psi:=\Phi^{-1}=\left(
\begin{array}{ccc}
1&0&0\\
-x^{1}/t&1&0\\
-x^{2}/t&0&1
\end{array}\right)
\ee
are homogeneous 0-th order matrices. These frames allow us to decompose an arbitrary tensor $T$,
and to set
$$
T=T^{\alpha\beta}\del_\alpha \otimes\del_\beta =T^{\alpha\beta}\Psi_\alpha ^{\alpha^{'}}
\Psi_\beta ^{\beta^{'}}\underline{\del}_{\alpha^{'}}\otimes\underline{\del}_{\beta^{'}}
:=\underline{T}^{\alpha^{'}\beta^{'}}\underline{\del}_{\alpha^{'}}\otimes\underline{\del}_{\beta^{'}},
$$
so $\underline{T}^{\alpha^{'}\beta^{'}}=T^{\alpha\beta}\Psi_\alpha ^{\alpha^{'}}\Psi_\beta ^{\beta^{'}}$. For instance, the Minkowski metric $m=\underm^{\alpha\beta}\underline{\del}_\alpha \otimes\underline{\del}_\beta$
can be expressed in the semi-hyperboloidal frame as
with
\be
(\underm^{\alpha\beta}) = \left(\begin{array}{ccc}
s^{2}/t^2&x^1/t&x^2/t\\
x^1/t&-1&0\\
x^2/t&0&-1
\end{array}
\right),
\qquad
(\underm_{\alpha\beta}) = \left(\begin{array}{ccc}
1&x^1/t&x^2/t\\
x^1/t&(x^1/t)^2-1&x^{1}x^{2}/t^{2}\\
x^2/t&x^{1}x^{2}/t^{2}&(x^{2}/t)^{2}-1
\end{array}
\right).
\ee
Similarly, a second-order differential operator $T^{\alpha\beta}\del_\alpha \del_\beta $ can also be written in the semi-hyperboloidal frame, we have the following decomposition formula
\bel{2.4}
T=\underline{T}^{\alpha\beta}\underline{\del}_\alpha \underline{\del}_\beta +T^{\alpha\beta}(\del_\alpha \Psi^{\beta^{'}}_\beta )
\underline{\del}_{\beta^{'}}.
\ee
Then for the classical wave operator $\Box$, we have
\begin{eqnarray}\label{2.5}
(s/t)^{2}\underline{\del}_{0}\underline{\del}_{0}=\Box&-&\underm^{0a}\underline{\del}_{0}\underline{\del}_{a}-
\underm^{a0}\underline{\del}_{a}\underline{\del}_{0}-\underm^{ab}\underline{\del}_{a}\underline{\del}_{b}\nonumber
-m^{\alpha\beta}(\del_\alpha \Psi^{\beta^{'}}_\beta )\underline{\del}_{\beta^{'}}.
\end{eqnarray}


\subsection*{Energy estimate for the hyperboloidal foliation}\label{Section2.2}

We establish the energy estimate for the quasilinear wave equation
\bel{2.6}
\Box\phi+G^{\alpha\beta}\del_\alpha \del_\beta \phi=F, 
\ee
where
$$
G^{\alpha\beta}=(g^{\alpha\beta}-m^{\alpha\beta})+g^{\mu\nu}g^{\alpha\beta}
\del_\mu \phi\del_{\nu}\phi
-g^{\alpha\mu}g^{\beta\nu}\del_\mu \phi\del_{\nu}\phi
$$
and
$$
F=\Gamma^\mu\del_\mu \phi+g^{\alpha\beta}\Gamma^\mu\del_\alpha \phi
\del_\beta \phi\del_\mu \phi
-g^{\mu\alpha}g^{\nu\beta}\Gamma_{\mu\nu}^{\gamma}\del_\alpha \phi
\del_\beta \phi\del_{\gamma}\phi.
$$
We set 
\bel{2.7}
\Et_{0}(s,\phi):=E_{0}(s,\phi)+2\int_{H_{s}}(G^{\alpha\beta}\del_{t}\phi\del_\beta \phi)\cdot(1,-x/t) \, dx
-\int_{H_{s}}(G^{\alpha\beta}\del_\alpha \phi\del_\beta \phi) \, dx,
\ee
where $E_{0}(s,\phi)$ is defined by \eqref{1.9}.
Multiplying both sides of \eqref{2.6} by $\del_{t}\phi$ and integrating by parts in $K_{[s_{0},s_{1}]}$, we obtain the energy identity
\bel{2.8}
\frac{d}{ds}\Et_{0}(s,\phi) = 2\int_{H_{s}}\left((s/t)\del_\alpha G^{\alpha\beta}\del_{t}\phi\del_\beta \phi
-(s/2t)\del_{t}G^{\alpha\beta}\del_\alpha \phi\del_\beta \phi+(s/t)\del_{t}\phi F\right) \, dx.
\ee

Similarly, for the higher order energy, we set 
\begin{eqnarray}\label{2.9}
\Et_{m}(s,\phi) = E_{m}(s,\phi)&+&\sum_{|I|\leq m}\Big(
2\int_{H_{s}}(G^{\alpha\beta}\del_{t}Z^{I}\phi\del_\beta Z^{I}\phi)\cdot(1,-x/t) \, dx\nonumber\\
&-&\int_{H_{s}}(G^{\alpha\beta}\del_\alpha Z^{I}\phi\del_\beta Z^{I}\phi) \, dx\Big).
\end{eqnarray}
Using $[\Box,Z^{I}]=0$, for $Z=(\del_\alpha ,L_{a})$, we have
\bel{2.10}
\Box Z^{I}\phi+G^{\alpha\beta}(\del\phi)\del_{\alpha\beta}^{2}Z^{I}\phi=Z^{I}F
-[Z^{I},G^{\alpha\beta}(\del\phi)\del_{\alpha\beta}^{2}]\phi:=\tilde{F}^{I}
\ee
and this leads us to the energy identity
\begin{eqnarray}\label{2.11}
\frac{d}{ds}\Et_{m}(s,\phi)&=&2\sum_{|I|\leq m}\int_{H_{s}}\Big((s/t)\del_\alpha G^{\alpha\beta}\del_{t}Z^{I}\phi\del_\beta Z^{I}\phi\nonumber\\
&& \hskip2.cm 
-(s/2t)\del_{t}G^{\alpha\beta}\del_\alpha Z^{I}\phi\del_\beta Z^{I}\phi+(s/t)\del_{t}Z^{I}\phi \tilde{F}^{I}\Big) \, dx.
\end{eqnarray}


\subsection*{Sobolev inequality on hyperboloids}\label{Section2.3}

The following Sobolev inequality is a natural generalization of a statement in \cite{Klainerman85,Hormander,PLF-YM-book}.

\begin{prop}\label{proposition2.1}
If $\phi$ be a sufficiently regular function defined in the cone $K=\{|x|<t-1\}$, then for all $s>0$ and with $t=\sqrt{s^{2}+|x|^{2}}$ one has 
\bel{2.12}
\sup_{H_{s}}t^{n/2}|\phi(t,x)|
\lesssim
\sum_{a}\sum_{|I|\leq\lceil n/2 \rceil}\|L_a^{I}\phi\|_{L^{2}(H_{s})},
\ee
where $\lceil y \rceil$ is the ceiling function and the summation is over all vector fields $L_{a}$.
\end{prop}


\section{The nonlinear structure in the semi-hyperboloidal frame}
\label{Section3}

\subsection*{Null forms}

We must express all the nonlinear terms in the semi-hyperboloidal frame. 
First of all, we consider the following null forms
\bel{3.1}
\aligned
&m^{\mu\nu}m^{\alpha\beta}\del_\alpha \phi\del_\beta \phi
\del^{2}_{\mu\nu}\phi-m^{\mu\alpha}m^{\nu\beta}\del_\alpha \phi\del_\beta \phi
\del^{2}_{\mu\nu}\phi\
\\
&=
-(\del_{t}\phi)^2\Delta\phi-|\nabla\phi|^{2}\del^{2}_{t}\phi
+|\nabla\phi|^{2}\Delta\phi+\sum_{i=1}^{2}2\del_{t}\phi\del_{i}\phi\del_{ti}\phi
-\sum_{i,j=1}^{2}\del_{i}\phi\del_{j}\phi\del_{ij}\phi.
\endaligned
\ee
By a straightforward but tedious calculation, we obtain the following identities
$$
\aligned
&-(\del_{t}\phi)^2\Delta\phi-|\nabla\phi|^{2}\del^{2}_{t}\phi
+\sum_{i=1}^{2}2\del_{t}\phi\del_{i}\phi\del_{ti}\phi
\\
&=-\sum_{i=1}^{2}(\underline{\del}_{t}\phi)^2(-\frac{x^{i}}{t}\underline
{\del}_{t}+\underline{\del}_{i})(-\frac{x^{i}}{t}\underline{\del}_{t}\phi
+\underline{\del}_{i}\phi)
-\sum_{i=1}^{2}(-\frac{x^{i}}{t}\underline{\del}_{t}\phi
+\underline{\del}_{i}\phi)^{2}(\underline{\del}_{t}^{2}\phi)
\\
&\quad +2\sum_{i=1}^{2}\underline{\del}_{t}\phi(-\frac{x^{i}}{t}\underline{\del}_{t}\phi
+\underline{\del}_{i}\phi)\underline{\del}_{t}
(-\frac{x^{i}}{t}\underline{\del}_{t}\phi+\underline{\del}_{i}\phi)
\\
&=\sum_{i=1}^{2}\Big(\frac{x^{i}}{t}(\underline{\del}_{t}\phi)^{2}
\underline{\del}_{i}\underline{\del}_{t}\phi
-\frac{x^{i}}{t}(\underline{\del}_{t}\phi)^{2}
\underline{\del}_{t}\underline{\del}_{i}\phi
-(\underline{\del}_{t}\phi)^{2}\underline{\del}_{i}\underline{\del}_{i}\phi
-(\underline{\del}_{i}\phi)^{2}\underline{\del}^{2}_{t}\phi
+2\underline{\del}_{t}\phi\underline{\del}_{i}\phi\underline{\del}_{t}
\underline{\del}_{i}\phi
\\
&\quad +\frac{2(t^2-r^2)}{t^{3}}(\underline{\del}_{t}\phi)^{3}
+2\frac{x^{i}}{t^{2}}(\underline{\del}_{t}\phi)^{2}\underline{\del}_{i}\phi\Big)
\endaligned
$$
and therefore 
$$
\aligned
&-(\del_{t}\phi)^2\Delta\phi-|\nabla\phi|^{2}\del^{2}_{t}\phi
+\sum_{i=1}^{2}2\del_{t}\phi\del_{i}\phi\del_{ti}\phi
\\
&=\frac{r^{2}}{t^{3}}(\underline{\del}_{t}\phi)^{3}+\sum_{i=1}^2\Big(-(\underline{\del}_{t}\phi)^{2}\underline{\del}_{i}\underline{\del}_{i}\phi
-(\underline{\del}_{i}\phi)^{2}\underline{\del}^{2}_{t}\phi
+2\underline{\del}_{t}\phi\underline{\del}_{i}\phi\underline{\del}_{t}
\underline{\del}_{i}\phi
  +\frac{2(t^2-r^2)}{t^{3}}(\underline{\del}_{t}\phi)^{3}
+2\frac{x^{i}}{t^{2}}(\underline{\del}_{t}\phi)^{2}\underline{\del}_{i}\phi\Big). 
\endaligned
$$
We have 
$$
\aligned
&|\nabla\phi|^{2}\Delta\phi-\sum_{i,j=1}^{2}\del_{i}\phi\del_{j}\phi\del_{ij}\phi
=(\sum_{i=1}^{2}(\del_{i}\phi)^{2})(\sum_{j=1}^{2}\del^{2}_{j}\phi)
-\sum_{i,j=1}^{2}\del_{i}\phi\del_{j}\phi\del_{ij}\phi
\\
&=(\del_{1}\phi)^2\del_{2}^{2}\phi+(\del_{2}\phi)^2\del_{1}^{2}\phi
-2\del_{1}\phi\del_{2}\phi\del_{1}\del_{2}\phi
\\
&=(-\frac{x^{1}}{t}\underline{\del}_{t}\phi+
\underline{\del}_{1}\phi)^2
(-\frac{x^{2}}{t}\underline{\del}_{t}+\underline{\del}_{2})
(-\frac{x^{2}}{t}\underline{\del}_{t}\phi+\underline{\del}_{2}\phi)
\\
&+(-\frac{x^{2}}{t}\underline{\del}_{t}\phi+
\underline{\del}_{2}\phi)^2
(-\frac{x^{1}}{t}\underline{\del}_{t}+\underline{\del}_{1})
(-\frac{x^{1}}{t}\underline{\del}_{t}\phi+\underline{\del}_{1}\phi)
\\
&-2(-\frac{x^{1}}{t}\underline{\del}_{t}\phi
+\underline{\del}_{1}\phi)
(-\frac{x^{2}}{t}\underline{\del}_{t}\phi
+\underline{\del}_{2}\phi)
(-\frac{x^{1}}{t}\underline{\del}_{t}+\underline{\del}_{1})
(-\frac{x^{2}}{t}\underline{\del}_{t}\phi+\underline{\del}_{2}\phi)
\endaligned
$$
and this leads us to the general identity 
$$
\aligned
&|\nabla\phi|^{2}\Delta\phi-\sum_{i,j=1}^{2}\del_{i}\phi\del_{j}\phi\del_{ij}\phi
=(\sum_{i=1}^{2}(\del_{i}\phi)^{2})(\sum_{j=1}^{2}\del^{2}_{j}\phi)
-\sum_{i,j=1}^{2}\del_{i}\phi\del_{j}\phi\del_{ij}\phi
\\
&=(\underline{\del}_{t}\phi)^2\Big(\frac{(x^1)^2}{t^2}
\underline{\del}_{2}^2\phi+\frac{(x^2)^2}{t^2}
\underline{\del}_{1}^2\phi
-2\frac{x^{1}x^{2}}{t^{2}}\underline{\del}_{1}\underline{\del}_{2}\phi\Big)
\\
&+\underline{\del}_{t}^2\phi\Big(
\frac{(x^{2})^2}{t^2}(\underline{\del}_{1}\phi)^2+
\frac{(x^{1})^2}{t^2}(\underline{\del}_{2}\phi)^2
-2\frac{x^{1}x^{2}}{t^2}\underline{\del}_{1}\phi\underline{\del}_{2}\phi\Big)
\\
&-\frac{x^{2}}{t}(\underline{\del}_{1}\phi)^2(\underline{\del}_{t}\underline{\del}_{2}\phi
-\underline{\del}_{2}\underline{\del}_{t}\phi)
-\frac{x^{1}}{t}(\underline{\del}_{2}\phi)^2(\underline{\del}_{t}\underline{\del}_{1}\phi
-\underline{\del}_{1}\underline{\del}_{t}\phi)
\\
&+2\frac{x^{1}}{t}\underline{\del}_{1}\phi\underline{\del}_{2}\phi
\underline{\del}_{t}\underline{\del}_{2}\phi
+2\frac{x^{2}}{t}\underline{\del}_{1}\phi\underline{\del}_{2}\phi
\underline{\del}_{1}\underline{\del}_{t}\phi
+(\underline{\del}_{1}\phi)^2\underline{\del}_{2}\underline{\del}_{2}\phi
+(\underline{\del}_{2}\phi)^2\underline{\del}_{1}\underline{\del}_{1}\phi
+2\underline{\del}_{1}\phi\underline{\del}_{2}\phi\underline{\del}_{1}\underline{\del}_{2}\phi
\\
&+2\frac{x^{1}x^{2}}{t^{2}}\underline{\del}_{t}\phi
\underline{\del}_{1}\phi\underline{\del}_{2}\underline{\del}_{t}\phi
+2\frac{x^{1}x^{2}}{t^{2}}\underline{\del}_{t}\phi
\underline{\del}_{2}\phi\underline{\del}_{t}\underline{\del}_{1}\phi
 -2\frac{(x^{1})^2}{t^2}\underline{\del}_{t}\phi\underline{\del}_{2}\phi
\underline{\del}_{t}\underline{\del}_{2}\phi
-2\frac{(x^{2})^2}{t^2}\underline{\del}_{t}\phi\underline{\del}_{1}\phi
\underline{\del}_{1}\underline{\del}_{t}\phi
\\
&-2\frac{x^{1}-x^{2}}{t}\underline{\del}_{t}\phi\underline{\del}_{1}\phi
\underline{\del}_{1}\underline{\del}_{2}\phi
+2\frac{x^{1}-x^{2}}{t}\underline{\del}_{t}\phi\underline{\del}_{2}\phi
\underline{\del}_{1}\underline{\del}_{2}\phi
\\
&-\frac{r^2}{t^{3}}(\underline{\del}_{t}\phi)^{3}
-\frac{1}{t}\underline{\del}_{t}\phi\Big((\underline{\del}_{1}\phi)^{2}
+(\underline{\del}_{2}\phi)^{2}\Big)
+(\underline{\del}_{t}\phi)^{2}\Big(
\frac{2x^{1}}{t^{2}}\underline{\del}_{1}\phi+
\frac{2x^{2}}{t^{2}}\underline{\del}_{2}\phi\Big).
\endaligned
$$
Combining the above two identities, we deduce that the nonlinearity of interest arising in the membrane equation reads 
\begin{eqnarray}\label{3.2}
&&-(\del_{t}\phi)^2\Delta\phi-|\nabla\phi|^{2}\del^{2}_{t}\phi
+|\nabla\phi|^{2}\Delta\phi+\sum_{i=1}^{2}2\del_{t}\phi\del_{i}\phi\del_{ti}\phi
-\sum_{i,j=1}^{2}\del_{i}\phi\del_{j}\phi\del_{ij}\phi\nonumber\\
&=&\sum_{i=1}^2\Big(-(\underline{\del}_{t}\phi)^{2}\underline{\del}_{i}\underline{\del}_{i}\phi
-(\underline{\del}_{i}\phi)^{2}\underline{\del}^{2}_{t}\phi
+2\underline{\del}_{t}\phi\underline{\del}_{i}\phi\underline{\del}_{t}
\underline{\del}_{i}\phi
\Big)\nonumber\\
&&+(\underline{\del}_{t}\phi)^2\Big(\frac{(x^1)^2}{t^2}
\underline{\del}_{2}^2\phi+\frac{(x^2)^2}{t^2}
\underline{\del}_{1}^2\phi
-2\frac{x^{1}x^{2}}{t^{2}}\underline{\del}_{1}\underline{\del}_{2}\phi\Big)\nonumber\\
&&+\underline{\del}_{t}^2\phi\Big(
\frac{(x^{2})^2}{t^2}(\underline{\del}_{1}\phi)^2+
\frac{(x^{1})^2}{t^2}(\underline{\del}_{2}\phi)^2
-2\frac{x^{1}x^{2}}{t^2}\underline{\del}_{1}\phi\underline{\del}_{2}\phi\Big)\nonumber
-\frac{x^{2}}{t}(\underline{\del}_{1}\phi)^2(\underline{\del}_{t}\underline{\del}_{2}\phi
-\underline{\del}_{2}\underline{\del}_{t}\phi)
-\frac{x^{1}}{t}(\underline{\del}_{2}\phi)^2(\underline{\del}_{t}\underline{\del}_{1}\phi
-\underline{\del}_{1}\underline{\del}_{t}\phi)\nonumber\\
&&+2\frac{x^{1}}{t}\underline{\del}_{1}\phi\underline{\del}_{2}\phi
\underline{\del}_{t}\underline{\del}_{2}\phi
+2\frac{x^{2}}{t}\underline{\del}_{1}\phi\underline{\del}_{2}\phi
\underline{\del}_{1}\underline{\del}_{t}\phi\nonumber
+(\underline{\del}_{1}\phi)^2\underline{\del}_{2}\underline{\del}_{2}\phi
+(\underline{\del}_{2}\phi)^2\underline{\del}_{1}\underline{\del}_{1}\phi
+2\underline{\del}_{1}\phi\underline{\del}_{2}\phi\underline{\del}_{1}\underline{\del}_{2}\phi\nonumber\\
&&+2\frac{x^{1}x^{2}}{t^{2}}\underline{\del}_{t}\phi
\underline{\del}_{1}\phi\underline{\del}_{2}\underline{\del}_{t}\phi
+2\frac{x^{1}x^{2}}{t^{2}}\underline{\del}_{t}\phi
\underline{\del}_{2}\phi\underline{\del}_{t}\underline{\del}_{1}\phi\nonumber\\
&&-2\frac{(x^{1})^2}{t^2}\underline{\del}_{t}\phi\underline{\del}_{2}\phi
\underline{\del}_{t}\underline{\del}_{2}\phi
-2\frac{(x^{2})^2}{t^2}\underline{\del}_{t}\phi\underline{\del}_{1}\phi
\underline{\del}_{1}\underline{\del}_{t}\phi\nonumber
-2\frac{x^{1}-x^{2}}{t}\underline{\del}_{t}\phi\underline{\del}_{1}\phi
\underline{\del}_{1}\underline{\del}_{2}\phi
+2\frac{x^{1}-x^{2}}{t}\underline{\del}_{t}\phi\underline{\del}_{2}\phi
\underline{\del}_{1}\underline{\del}_{2}\phi\nonumber\\
&&
-\frac{1}{t}\underline{\del}_{t}\phi\Big((\underline{\del}_{1}\phi)^{2}
+(\underline{\del}_{2}\phi)^{2}\Big)
+(\underline{\del}_{t}\phi)^{2}\Big(
\frac{2x^{1}}{t^{2}}\underline{\del}_{1}\phi+
\frac{2x^{2}}{t^{2}}\underline{\del}_{2}\phi\Big)\nonumber
+\frac{2(t^2-r^2)}{t^{3}}(\underline{\del}_{t}\phi)^{3}
+2\frac{x^{i}}{t^{2}}(\underline{\del}_{t}\phi)^{2}\underline{\del}_{i}\phi, 
\end{eqnarray}
which with bvious notation takes the following general form
$$
T^{\alpha\beta\delta\gamma}\del_\alpha \phi\del_\beta \phi\del_{\gamma}\del_{\delta}\phi.
$$
We have thus arrived at the following conclusion. 

\begin{lemma}\label{Lemma3.1}
For the null cubic form $T^{\alpha\beta\delta\gamma}\del_\alpha \phi
\del_\beta \phi\del_{\gamma}\del_{\delta}\phi$ with constant coefficients $T^{\alpha\beta\delta\gamma}$ satisfying the above null condition and for any index $I$, one has
$$
\aligned
& |Z^{I}(T^{\alpha\beta\delta\gamma}\del_\alpha \phi\del_\beta \phi\del_{\gamma}\del_{\delta}\phi)|
\\
&\lesssim (s/t)^{2}t^{-1}\sum_{|I_{1}|+|I_{2}|+|I_{3}|\leq|I|}
|Z^{I_{1}}\underline{\del}_{t}\phi Z^{I_{2}}\underline{\del}_{t}\phi
Z^{I_{3}}\underline{\del}_{t}\phi|
 +t^{-1}\sum_{|I_{1}|+|I_{2}|+|I_{3}|\leq|I|}
|Z^{I_{1}}\underline{\del}_\alpha \phi Z^{I_{2}}\underline{\del}_\beta \phi
Z^{I_{3}}\underline{\del}_{i}\phi|
\\
&+\sum_{|I_{1}|+|I_{2}|+|I_{3}|\leq|I|}
|Z^{I_{1}}\underline{\del}_\alpha \phi Z^{I_{2}}\underline{\del}_\beta \phi
Z^{I_{3}}\underline{\del}_{i}\underline{\del}_{j}\phi|
 +\sum_{|I_{1}|+|I_{2}|+|I_{3}|\leq|I|}
|Z^{I_{1}}\underline{\del}_{j}\phi Z^{I_{2}}\underline{\del}_\beta \phi
Z^{I_{3}}\underline{\del}_{i}\underline{\del}_\alpha \phi|. 
\endaligned
$$
\end{lemma}

At this stage, we see that the nonlinear terms having more than three derivatives with respect to $\underline{\del}_{t}$ contain the {\sl favorable coefficient} $(s/t)^{2}$ when expressed in the semi-hyperboloidal frame. Moreover, the term $\underline{\del}_{t}\underline{\del}_{t}\phi$ {\sl does not arise}. We will refer to this nonlinear structure as the ``double null condition'', and it will pla an essential role in the following estimates.


\subsection*{General nonlinear terms}

We thus have to estimate the following general nonlinear terms
\bel{3.3}
(g^{\mu\nu}-m^{\mu\nu})g^{\alpha\beta}\del_\alpha \phi\del_\beta \phi
\del_\mu \del_{\nu}\phi,
\ee
\bel{3.4}
(g^{\mu\nu}-m^{\mu\nu})(g^{\alpha\beta}-m^{\alpha\beta})\del_\alpha \phi\del_\beta \phi
\del_\mu \del_{\nu}\phi, 
\ee
and
\be
g^{\alpha\beta}\Gamma^\mu\del_\alpha \phi\del_\beta \phi\del_{\gamma}\phi,
\quad
g^{\mu\alpha}g^{\nu\beta}\Gamma_{\mu\nu}^{\gamma}\del_\alpha \phi
\del_\beta \phi\del_{\gamma}\phi.
\ee
From our assumptions on the background metric, we can obtain some decay and so we rewrite the above terms as
$$
D^{\alpha\beta\gamma\delta}\del_\alpha \phi\del_\beta \phi
\del_{\delta}\del_{\gamma}\phi,
\qquad D^{\alpha\beta\gamma}\del_\alpha \phi\del_\beta \phi\del_{\gamma}\phi,
$$
in which 
the coefficients $D^{\alpha\beta\gamma\delta}$ enjoy the following decay property: 
\be
|Z^{I}D^{\alpha\beta\gamma\delta}|\leq C\delta(s/t)^{2}s^{-1-\gamma},
\qquad
|Z^{I}D^{\alpha\beta\gamma}|\leq C\delta(s/t)^{2}s^{-1-\gamma}.
\ee
We thus have the following estimates on the above nonlinear terms.

\begin{lemma}\label{Lemma3.3}
Under the assumptions made on the background metric, one has 
$$
\aligned
& |Z^{I}(D^{\alpha\beta\gamma\delta}\del_\alpha \phi
\del_\beta \phi\del_{\delta}\del_{\gamma}\phi)|
\\
&\lesssim
\sum_{|I_{1}+|I_{2}|+|I_{3}|\leq|I|} \delta(s/t)^{2}s^{-1-\gamma}
\Big(|Z^{I_{1}}\underline{\del}_\alpha \phi
Z^{I_{2}}\underline{\del}_\beta \phi Z^{I_{3}}\underline{\del}_{\delta}\underline{\del}_{\gamma}\phi|
+t^{-1}|Z^{I_{1}}\underline{\del}_\alpha \phi
Z^{I_{2}}\underline{\del}_\beta \phi Z^{I_{3}}\underline{\del}_{\delta}\phi|\Big)
\endaligned
$$
and
\be
|Z^{I}(D^{\alpha\beta\gamma}\del_\alpha \phi
\del_\beta \phi\del_{\gamma}\phi)|
\lesssim
 \sum_{|I_{1}+|I_{2}|+|I_{3}|\leq|I|} \delta(s/t)^{2}s^{-1-\gamma}
|Z^{I_{1}}\underline{\del}_\alpha \phi
Z^{I_{2}}\underline{\del}_\beta \phi Z^{I_{3}}\underline{\del}_{\gamma}\phi|.
\ee
\end{lemma}


\subsection*{Linear terms with variable coefficients}

At last, we have to consider the following linear terms
\be
(g^{\mu\nu}-m^{\mu\nu})\del_\mu \del_{\nu}\phi,
\qquad
\Gamma^\mu\del_\mu \phi.
\ee
From our decay assumptions on $g^{\mu\nu}-m^{\mu\nu}$ and $\del g^{\mu\nu}$, we have the following lproperty.  

\begin{lemma}\label{Lemma3.4}
Under the assumptions made on the background metric, one has 
\be
|Z^{I}\Big((g^{\mu\nu}-m^{\mu\nu})\del_\mu \del_{\nu}\phi\Big)|
\lesssim
 \sum_{|J|\leq|I|} \delta(s/t)^{2}s^{-1-\gamma}\Big(|Z^{J}\underline{\del}_\alpha
\underline{\del}_\beta \phi|+t^{-1}Z^{J}\underline{\del}_\alpha \phi\Big),
\ee 
\be
|Z^{I}(\Gamma^\mu\del_\mu \phi)|
\lesssim
\sum_{|J|\leq |I|} \delta(s/t)^{2}s^{-1-\gamma}|Z^{J}\underline{\del}_\alpha \phi|.
\ee
\end{lemma}


\section{The $L^\infty$ and $L^2$ estimates}
\label{Section4}

Based on the Sobolev inequality and the definition of the energy, we have the following result. 

\begin{lemma}[$L^\infty$ estimates on derivatives up to first-order]
\label{Lemma4.1}
If $\phi$ is a sufficiently regular function supported in $K$, then the following estimates hold:
$$
\aligned
\sup_{H_{s}}|s\del_\alpha \phi|
& \lesssim \sum_{|I|\leq2}E_{0}^{1/2}(s,Z^{I}\phi),
\\
\sup_{{H_{s}}}|t\underline{\del}_{a}\phi|
&\lesssim \sum_{|I|\leq2}E_{0}^{1/2}(s,Z^{I}\phi),
\endaligned
$$
where the sum is over all admissible vector fields $Z$ in the list $\del_\alpha, L_{a}$.
\end{lemma}

We thus see that, for the ``good'' derivatives we get a decay rate of $t^{-1}$, while for the ``bad'' derivatives, we only get a decay $s^{-1}=\frac{1}{\sqrt{(t+r)(t-r)}}$ so that near the lightcone, we only get $t^{-1/2}$.
On the other hand, for the second-order derivatives with at least one good derivative $\underline{\del}_{a}$, we have the following result, which shows that the good second-order derivatives has an extra decay $t^{-1}$. 
Here, we use $\underline{\del}_{a}=t^{-1}L_{a}$.

\begin{lemma}[Bounds on the second order derivatives]
\label{Lemma4.3}
 For every sufficiently regular function $\phi$ supported in the cone $K$, the following estimates hold:
\be
\sup_{H_{s}}|ts\underline{\del}_{a}\underline{\del}_\alpha \phi|
+\sup_{H_{s}}|ts\underline{\del}_\alpha \underline{\del}_{a}\phi|
\lesssim \sum_{|I|\leq3}E^{1/2}_{0}(s, Z^{I}u),
\ee
\be
\int_{H_{s}}|s\underline{\del}_{a}\underline{\del}_\alpha Z^{I-1}u|^{2}dx
+\int_{H_{s}}|s\underline{\del}_\alpha \underline{\del}_{a}Z^{I-1}u|^{2}dx
\lesssim \sum_{|I|}E_{0}(s,Z^{I}u).
\ee
\end{lemma}

In view of the semi-hyperboloidal decomposition of the wave operator $\Box$, we easily obtain the following property. 

\begin{lemma}\label{Lemma4.5}
If $\phi$ is a smooth function compactly supported in the half cone $K=\{|x|\leq t-1\}$, then for any index $I$ one has
\be
|\del_\alpha \del_\beta Z^{I}\phi|
\lesssim
 |\del_{t}\del_{t}Z^{I}\phi|
+\sum_{a,\beta}|\underline{\del}_{a}\del_\beta Z^{I}\phi|+
\frac{1}{t}\sum_{\gamma}|\del_{\gamma}Z^{I}\phi|.
\ee
\end{lemma}

Next, based on \eqref{2.10}, we immediately deduce the following result. 

\begin{lemma}\label{Lemma4.6}
The following identity holds
\be
\Big(\Big(\frac{s}{t} \Big)^2+\underline{G}^{00}\Big)\del_{t}^{2}Z^{I}\phi
=\tilde{F}^{I}-Q^{I}+R(Z^{I}\phi),
\ee
where
$$
\aligned
\underline{G}^{\alpha\beta}
& =(g^{\alpha\beta}-m^{\alpha\beta})
+g^{\mu\nu}g^{\alpha\beta}\Psi_\mu ^{\alpha^{'}}\Psi_{\nu}^{\beta^{'}}
\underline{\del}_{\alpha^{'}}\phi\underline{\del}_{\beta^{'}}\phi
-g^{\mu\alpha}g^{\nu\beta}\Psi_\mu ^{\alpha^{'}}\Psi_{\nu}^{\beta^{'}}
\underline{\del}_{\alpha^{'}}\phi\underline{\del}_{\beta^{'}}\phi,
\\
Q^{I}
& =\underline{G}^{0a}\underline{\del}_{0}\underline{\del}_{a}Z^{I}\phi
+\underline{G}^{a0}\underline{\del}_{a}\underline{\del}_{0}Z^{I}\phi
+\underline{G}^{ab}\underline{\del}_{a}\underline{\del}_{b}Z^{I}\phi
+\underline{G}^{\alpha\beta}
\underline{\del}_\alpha \Psi^{\beta^{'}}_\beta \underline{\del}_{\beta^{'}}Z^{I}\phi,
\\
R(Z^{I}\phi) 
& = -\underm^{0a}\underline{\del}_{0}\underline{\del}_{a}Z^{I}\phi-
\underm^{a0}\underline{\del}_{a}\underline{\del}_{0}Z^{I}\phi
-\underm^{ab}\underline{\del}_{a}\underline{\del}_{b}Z^{I}\phi
+m^{\alpha\beta}(\del_\alpha \Psi^{\beta^{'}}_\beta )\underline{\del}_{\beta^{'}}Z^{I}\phi.
\endaligned
$$
\end{lemma}


Next, assume that for some large $|I|\geq 6$, there exists a large constant $C$ such that
\bel{5.1}
E^{1/2}_{|I|}(s,\phi)\leq C\epsilon.
\ee
Then we have the following $L^\infty$ estimate
\begin{lemma}\label{Lemma5.1}
Assume that $|J|\leq |I|-2$, and $|J^{+}|\leq |I|-3$, we have
\begin{eqnarray}
|s\del_\alpha Z^{J}\phi|&\leq&C\epsilon,\\
|t\underline{\del}_{a}Z^{J}\phi|&\leq&C\epsilon,\\
|ts\underline{\del}_{a}\underline{\del}_\alpha Z^{J^{+}}\phi|
+|ts\underline{\del}_\alpha \underline{\del}_{a}Z^{J^{+}}\phi|
&\leq&C\epsilon.
\end{eqnarray}
\end{lemma}

For $\del^{2}_{t}Z^{I}\phi$, we obtain the following result. 

\begin{lemma}\label{Lemma5.2}
\begin{eqnarray}
|\del^{2}_{t}Z^{I}\phi|&\leq& C\delta s^{-1-\gamma}(1+\sum_{|I_{1}|+|I_{2}|+|I_{3}|\leq|I|,|I_{1}|<|I|}|Z^{I_{2}}\del_{\gamma}\phi|
|Z^{I_{3}}\del_{\delta}\phi|)
|\del_{t}^{2}Z^{I_{1}}\phi|\nonumber\\
&&+C(\frac{t}{s})^{2}(\delta(s/t)^{2}s^{-1-\gamma}|M|+|Q^{I}|+|R(Z^{I}\phi)|),
\end{eqnarray}
where
$$
M=\sum_{|I_{1}|+|I_{2}|+|I_{3}|\leq |I|,|I_{1}|<|I|}|Z^{I_{2}}\del_{\gamma}\phi||Z^{I_{3}}\del_{\delta}\phi|
(|\underline{\del}_{a}\underline{\del}_\alpha Z^{I_{1}}\phi|
+t^{-1}|\del_\beta Z^{I_{1}}\phi|).
$$
\end{lemma}

\begin{proof} First, we need to check that
$
\Big(\frac{s}{t} \Big)^2+\underline{G}^{00}\sim \Big(\frac{s}{t} \Big)^2.
$
But this is clear since
$$
|\underline{G}^{00}|=|(g^{00}-m^{00})+D^{\mu\nu00}\Psi_\mu ^{\alpha}
\Psi_{\nu}^{\beta}\underline{\del}_\alpha \phi
\underline{\del}_\beta \phi|\leq
\delta(s/t)^{2}s^{-1-\gamma}(1+(C\epsilon)^{2}s^{-2}).
$$
Then for the lemma, we need to estimate $[Z^{I},\underline{G}^{00}\del_{t}^{2}]\phi$. We have
\bel{3.8}
[Z^{I},\underline{G}^{00}\del_{t}^{2}]\phi=
\sum_{|I_{1}|+|I_{2}|=|I|,|I_{1}|<|I|}Z^{I_{2}}
(\underline{G}^{00})Z^{I_{1}}(\del_{t}^{2}\phi)
+\underline{G}^{00}[Z^{I},\del_{t}^{2}]\phi
\ee
and
\be
|Z^{I_{2}}\underline{G}^{00}|\leq \delta (s/t)^{2}s^{-1-\gamma}\Big(1+\sum_{|I_{3}|+|I_{4}|\leq|I_{2}|}|Z^{I_{3}}\del_{\gamma}\phi|
|Z^{I_{4}}\del_{\delta}\phi|\Big).
\ee
For the second term in \eqref{3.8}, we have
\begin{eqnarray*}
|[Z^{I},\del^{2}_{t}]\phi|&\leq& \sum_{|J|<|I|,\alpha,\beta}|\del_\alpha \del_\beta Z^{J}\phi|\\
&+&C\sum_{|J|<|I|}|\del^{2}_{t}Z^{J}\phi|+C\sum_{|J|<|I|,a,\beta}
|\underline{\del}_{a}\underline{\del}_\beta Z^{J}\phi|
+Ct^{-1}\sum_{|J|<|I|,\gamma}|\del_{\gamma}Z^{J}\phi|.
\end{eqnarray*}
From the above inequalities, we arrive at the statement in the lemma.
\end{proof}

For the energy $E_{m}(s,\phi)$ and the generalized energy $\Et_{m}(s,\phi)$, we have the following equivalence.

\begin{lemma}\label{Lemma5.3}
Assume that the a priori assumption $E^{1/2}_{m}(s,\phi)\leq C\epsilon$ holds, we have
\be
\Et_{m}(s,\phi \simeq E_{m}(s,\phi).
\ee
\begin{proof}
From \eqref{2.9}, we have to estimate
$$
2\int_{H_{s}}(G^{\alpha\beta}\del_{t}Z^{I}\phi\del_\beta Z^{I}\phi)\cdot(1,-x/t) \, dx
-\int_{H_{s}}(G^{\alpha\beta}\del_\alpha Z^{I}\phi\del_\beta Z^{I}\phi) \, dx. 
$$
Obviously we have
$$
\aligned
&\Big|\int_{H_{s}}G^{\alpha\beta}\del_\alpha Z^{I}\phi\del_\beta Z^{I}\phi dx\Big|
\\
&=
\Big|\int_{H_{s}}\Big((g^{\alpha\beta}-m^{\alpha\beta})+g^{\mu\nu}g^{\alpha\beta}
\del_\mu \phi\del_{\nu}\phi-g^{\alpha\mu}g^{\nu\beta}
\del_\mu \phi\del_{\nu}\phi\Big)\del_\alpha Z^{I}\phi\del_\beta Z^{I}\phi dx\Big|
\\
&\leq\Big|\int_{H_{s}}(g^{\alpha\beta}-m^{\alpha\beta})\Psi_\alpha ^{\alpha^{'}}
\Psi_\beta ^{\beta^{'}}\underline{\del}_{\alpha^{'}}Z^{I}\phi\underline{\del}_{\beta^{'}}Z^{I}\phi dx\Big|
\\
& \quad +\Big|\int_{H_{s}}\Big(g^{\mu\nu}g^{\alpha\beta}
\del_\mu \phi\del_{\nu}\phi-g^{\alpha\mu}g^{\nu\beta}
\del_\mu \phi\del_{\nu}\phi\Big)\del_\alpha Z^{I}\phi\del_\beta Z^{I}\phi dx\Big|
=: I+II.
\endaligned
$$
For $I$, we have
$$
I\leq\Big|\int_{H_{s}}\delta(s/t)^{2}s^{-1-\gamma}\underline{\del}_\alpha Z^{I}\phi
\underline{\del}_\beta Z^{I}\phi dx\Big|\leq \delta s^{-1-\gamma}E_{m}(s,\phi),
$$
while for $II$ we consider first the following term
$$
III:=\Big|\int_{H_{s}}\Big(m^{\mu\nu}m^{\alpha\beta}\del_\mu \phi\del_{\nu}\phi
\del_\alpha Z^{I}\phi\del_\beta Z^{I}\phi-
m^{\mu\alpha}m^{\nu\beta}\del_\mu \phi\del_{\nu}\phi
\del_\alpha Z^{I}\phi\del_\beta Z^{I}\phi\Big) \, dx\Big|.
$$
Clearly, for $II-III$, we have
$$
|II-III|\leq C\delta\epsilon^{2} s^{-1-\gamma}E_{m}(s,\phi).
$$
For $II$, we use the null structure. Let us set
$v := Z^{I}\phi$.
We have in the semi-hyperboloidal frame
\begin{eqnarray*}
&&m^{\mu\nu}m^{\alpha\beta}\del_\mu \phi\del_{\nu}\phi\del_\alpha v\del_\beta v
-m^{\mu\alpha}m^{\nu\beta}\del_\mu \phi\del_{\nu}\phi\del_\alpha v\del_\beta v\\
&=&-(\del_{t}\phi)^{2}\left((\del_{1}v)^{2}+(\del_{2}v)^{2}\right)
-(\del_{t}v)^{2}\left((\del_{1}\phi)^{2}+(\del_{2}\phi)^{2}\right)
+2\del_{t}\phi\del_{t}v(\del_{1}\phi\del_{1}v+\del_{2}\phi\del_{2}v)\\
&&+(\del_{1}\phi)^{2}(\del_{2}v)^{2}+(\del_{2}\phi)^{2}(\del_{1}v)^{2}
-2\del_{1}\phi\del_{2}\phi\del_{1}v\del_{2}v\\
&=&[(\frac{x^{2}}{t}\underline{\del}_{1}\phi-
\frac{x^{1}}{t}\underline{\del}_{2}\phi)^{2}
-(\underline{\del}_{1}\phi)^{2}-(\underline{\del}_{2}\phi)^{2}](\underline{\del}_{t}v)^{2}\\
&&+[-2\frac{x^{1}}{t}(\underline{\del}_{2}\phi)^{2}
+2\frac{x^{1}x^{2}}{t^{2}}\underline{\del}_{2}\phi\underline{\del}_{t}\phi
-2\frac{(x^{2})^{2}}{t^{2}}\underline{\del}_{1}\phi\underline{\del}_{t}\phi
+2\frac{x^{2}}{t}\underline{\del}_{1}\phi\underline{\del}_{2}\phi
+2\underline{\del}_{t}\phi\underline{\del}_{1}\phi]
\underline{\del}_{t}v\underline{\del}_{1}v\\
&&+[-2\frac{x^{2}}{t}(\underline{\del}_{1}\phi)^{2}
+2\frac{x^{1}x^{2}}{t^{2}}\underline{\del}_{1}\phi\underline{\del}_{t}\phi
-2\frac{(x^{1})^{2}}{t^{2}}\underline{\del}_{2}\phi\underline{\del}_{t}\phi
+2\frac{x^{1}}{t}\underline{\del}_{1}\phi\underline{\del}_{2}\phi
+2\underline{\del}_{t}\phi\underline{\del}_{2}\phi]
\underline{\del}_{t}v\underline{\del}_{2}v\\
&&+[-(\underline{\del}_{t}\phi)^{2}+(\frac{x^{2}}{t}\underline{\del}_{t}\phi
-\underline{\del}_{2}\phi)^{2}](\underline{\del}_{1}v)^{2}
+[-(\underline{\del}_{t}\phi)^{2}+(\frac{x^{1}}{t}\underline{\del}_{t}\phi
-\underline{\del}_{1}\phi)^{2}](\underline{\del}_{2}v)^{2}\\
&&-[2\frac{x^{1}x^{2}}{t^{2}}(\underline{\del}_{t}\phi)^{2}
-2\frac{x^{1}}{t}\underline{\del}_{t}\phi\underline{\del}_{2}\phi
-2\frac{x^{2}}{t}\underline{\del}_{t}\phi\underline{\del}_{1}\phi
+2\underline{\del}_{1}\phi\underline{\del}_{2}\phi]\underline{\del}_{1}v\underline{\del}_{2}v.
\end{eqnarray*}
Thus, we have
\bel{5.8}
II\leq\frac{C^{2}\epsilon^{2}}{s^{2}}E_{m}(s,\phi).
\ee
Then the lemma holds by combining above analysis.
\end{proof}
\end{lemma}

In the following, we focus on the $L^\infty$ decay of $\underline{\del}_{t}^{2}Z^{J}\phi$ for $|J|\leq|I|-4$.

At first, we easily have
\begin{lemma}\label{Lemma5.4}
When $|J|+2\leq|I|-2$, we have
\begin{eqnarray}
|Q^{I}|&\leq& \frac{C\epsilon \delta}{st}\Big(\frac{s}{t} \Big)^2s^{-1-\gamma}+
\frac{C^{3}\epsilon^3}{s^{3}t},\\
|R(Z^{J}\phi)|&\leq& C\epsilon t^{-1}s^{-1},\\
|M|&\leq&C\epsilon^{3}t^{-1}s^{-3}.
\end{eqnarray}
\end{lemma}
Combining Lemmas \ref{Lemma5.2} and \ref{Lemma5.4}, by simple induction, we get
\begin{lemma}\label{Lemma5.5}
When $|J|+2\leq|I|-2$, we have
\be
|\del_{t}^{2}Z^{J}\phi|\leq C\epsilon ts^{-3}.
\ee
\end{lemma}


\section{The uniform bound on the energy}
\label{Section6}

In this section, we do energy estimates and close the a priori assumption \eqref{5.1}.
By \eqref{2.11}, for $|I|\leq m$, we have
\begin{eqnarray}\label{6.1}
\frac{d}{ds}\Et_{m}(s,\phi)&\leq&(\|\del_\alpha G^{\alpha\beta}\del_\beta Z^{I}\phi\|_{L^{2}(H_{s}
)}+\|\tilde{F}^{I}\|_{L^{2}(H_{s})})
E^{1/2}_{m}(s,\phi)\nonumber\\
&&+\Big|\int_{H_{s}}\frac{s}{t}\del_{t}G^{\alpha\beta}
\del_\alpha Z^{I}\phi\del_\beta Z^{I}\phi dx\Big|.
\end{eqnarray}
Thus, what we have to do is to estimate the following terms
\begin{eqnarray}\label{6.2}
&&\|\del_\alpha G^{\alpha\beta}\del_\beta Z^{I}\phi\|_{L^{2}(H_{s}
)},\\\label{6.3}
&&\|\tilde{F}^{I}\|_{L^{2}(H_{s})},\\\label{6.4}
&&\Big|\int_{H_{s}}\frac{s}{t}\del_{t}G^{\alpha\beta}
\del_\alpha Z^{I}\phi\del_\beta Z^{I}\phi dx\Big|.
\end{eqnarray}
We use several lemmas to estimate \eqref{6.2}-\eqref{6.4}.
\begin{lemma}\label{Lemma6.1}
We have
\be
\|\del_\alpha G^{\alpha\beta}\del_\beta Z^{I}\phi\|_{L^{2}(H_{s})}
\leq (C\delta s^{-1-\gamma}+C\epsilon^{2}s^{-3})E_{0}^{1/2}(s,Z^{I}\phi).
\ee
\end{lemma}
\begin{proof}
Direct calculation gives us
\begin{eqnarray*}
\del_\alpha G^{\alpha\beta}\del_\beta Z^{I}\phi&=&
\del_\alpha \Big(
(g^{\alpha\beta}-m^{\alpha\beta})+g^{\mu\nu}g^{\alpha\beta}\del_\mu \phi
\del_{\nu}\phi-g^{\alpha\mu}g^{\beta\nu}
\del_\mu \phi\del_{\nu}\phi\Big)
\del_\beta Z^{I}\phi\\
&=&\del_\alpha g^{\alpha\beta}\del_\beta Z^{I}\phi
+\del_\alpha \Big( m^{\mu\nu}m^{\alpha\beta}\del_\mu \phi
\del_{\nu}\phi- m^{\alpha\mu}m^{\beta\nu}
\del_\mu \phi\del_{\nu}\phi\Big)\del_\beta Z^{I}\phi\\
&&+\del_\alpha (D^{\alpha\beta\delta\gamma}\del_{\delta}\phi
\del_{\gamma}\phi)\del_\beta Z^{I}\phi\\
&:=&I+II+III.
\end{eqnarray*}
For $I$, by the assumption on the metric, we have
\bel{6.6}
\|I\|_{L^{2}(H_{s})}\leq C\delta\frac{1}{s^{1+\gamma}}\|(s/t)^{2}\del_\beta Z^{I}\phi\|_{L^{2}(H_{s})}
\leq C\delta\frac{1}{s^{1+\gamma}}E^{1/2}_{m}(s,\phi).
\ee
For $II$, we have to use the null condition, by Lemma \ref{Lemma3.1}, we have
\begin{eqnarray}\label{6.7}
\|II\|_{L^{2}(H_{s})}&=&\|T^{\alpha\beta\gamma\delta}
\del_{\alpha\beta}^{2}\phi\del_{\gamma}\phi\del_{\delta}Z^{I}\phi\|_{L^{2}(H_{s})}\nonumber\\
&\leq&\|C\Big((s/t)^{2}t^{-1}
|\underline{\del}_{t}\phi
\underline{\del}_{t}\phi Z^{I}\underline{\del}_{t}\phi|
+t^{-1}
|\underline{\del}_\alpha \phi
\underline{\del}_{i}\phi Z^{I}\underline{\del}_\beta \phi|\nonumber\\
&&+
|\underline{\del}_\alpha \phi
\underline{\del}_{i}\underline{\del}_{j}\phi Z^{I}\underline{\del}_\beta \phi|
+
|\underline{\del}_{j}\phi
\underline{\del}_{i}\underline{\del}_\alpha \phi Z^{I}\underline{\del}_\beta \phi|\Big)\|_{L^{2}(H_{s})}
\leq 
C\epsilon^{2}s^{-3}E^{1/2}_{m}(s,\phi).
\end{eqnarray}
For $III$, we have
\bel{6.8}
\|III\|_{H_{s}}\leq C\delta\|(s/t)^{2}s^{-1-\gamma}
\underline{\del}_\alpha \underline{\del}_\beta \phi
\underline{\del}_{\delta}\phi\underline{\del}_{\gamma}Z^{I}\phi\|_{L^{2}(H_{s})}
\leq C\delta\epsilon^{2}s^{-4-\gamma}E^{1/2}_{m}(s,\phi).
\ee
Combining \eqref{6.6}-\eqref{6.8}, we get
$$
\|\del_\alpha G^{\alpha\beta}\del_\beta Z^{I}\phi\|_{L^{2}(H_{s})}
\leq (C\delta s^{-1-\gamma}+C\epsilon^{2}s^{-3})E^{1/2}_{m}(s,\phi).
\qedhere $$
\end{proof}

Next we estimate \eqref{6.3}. From \eqref{2.11}, we have
\bel{6.9}
\tilde{F}^{I}=Z^{I}F-[Z^{I},G^{\alpha\beta}\del_{\alpha\beta}^{2}]\phi.
\ee
For the first term, we have
\bel{6.10}
\aligned
& \|Z^{I}F\|_{L^{2}(H_{s})}
\\
&\leq C\delta s^{-1-\gamma}\Big(\|\Big(\frac{s}{t} \Big)^2\underline{\del}_\alpha Z^{I}\phi\|_{L^{2}(H_{s})}
+\sum_{|I_{1}|+|I_{2}|+|I_{3}|\leq|I|}\|\Big(\frac{s}{t} \Big)^2
\underline{\del}_\alpha Z^{I_{1}}\phi
\underline{\del}_\alpha Z^{I_{2}}\phi
\underline{\del}_\alpha Z^{I_{3}}\phi\|_{L^{2}(H_{s})}\Big)
\\
&\leq C\delta s^{-1-\gamma}(1+C\epsilon^{2})E_{m}^{1/2}(s,\phi).
\endaligned
\ee
For the second term, we have
\begin{eqnarray*}
[Z^{I},G^{\alpha\beta}\del_{\alpha\beta}^{2}\phi]
&=&[Z^{I},(g^{\alpha\beta}-m^{\alpha\beta})\del^{2}_{\alpha\beta}\phi]
+[Z^{I},T^{\mu\nu\alpha\beta}\del_\mu \phi\del_{\nu}\phi
\del^{2}_{\alpha\beta}\phi]\\
&&+[Z^{I},D^{\mu\nu\alpha\beta}\del_\mu \phi\del_{\nu}\phi
\del^{2}_{\alpha\beta}\phi]:=A+B+I.
\end{eqnarray*}
We can easily obtain
\bel{6.11}
\|A\|_{L^{2}(H_{s})}\leq\sum_{|J|\leq|I|-1} C\delta s^{-1-\gamma}\|\Big(\frac{s}{t} \Big)^2\underline{\del}_\alpha \underline{\del}_\beta Z^{J}\phi\|
\leq C\delta s^{-1-\gamma}E_{m}^{1/2}(s,\phi).
\ee
Based on Lemma \ref{Lemma3.1}, we have
\bel{6.12}
\|B\|_{L^{2}(H_{s})}\leq C\epsilon^{2}s^{-3}E_{m}^{1/2}(s,\phi).
\ee
For $I$, we have
\begin{eqnarray}\label{6.13}
\|I\|_{L^{2}(H_{s})}&\leq& C\delta s^{-1-\gamma}
\sum_{|I_{1}|+|I_{2}|+|I_{3}|\leq|I|,|I_{3}<|I||}
\|(s/t)^{2}\underline{\del}_\alpha Z^{I_{1}}\phi
\underline{\del}_\beta Z^{I_{2}}\phi
\underline{\del}_{\delta}\underline{\del}_{\gamma}Z^{I_{3}}\phi\|_{L^{2}(H_{s})}\nonumber\\
&\leq&C\delta\epsilon^{2}s^{-3-\gamma}E_{m}^{1/2}(s,\phi).
\end{eqnarray}

Combining \eqref{6.10}--\eqref{6.13}, we reach the following conclusion. 

\begin{lemma}\label{Lemma6.2}
We have
\bel{6.14}
\|\tilde{F}^{I}\|_{L^{2}(H_{s})}\leq (C\delta s^{-1-\gamma}+C\epsilon^{2}s^{-3})E_{m}^{1/2}(s,\phi).
\ee
\end{lemma}
At last, we estimate \eqref{6.4}
\bel{6.15}
\aligned
\Big|\int_{H_{s}}\frac{s}{t}\del_{t}G^{\alpha\beta}
\del_\alpha Z^{I}\phi\del_\beta Z^{I}\phi\Big|
& =
\Big|\int_{H_{s}}\frac{s}{t}\del_{t}(g^{\alpha\beta}-m^{\alpha\beta})
\del_\alpha Z^{I}\phi\del_\beta Z^{I}\phi dx
\\
&+\int_{H_{s}}\del_{t}(T^{\alpha\beta\mu\nu}\del_\mu \phi\del_{\nu}\phi)
\del_\alpha Z^{I}\phi\del_\beta Z^{I}\phi dx
\\
&+
\int_{H_{s}}\del_{t}(D^{\alpha\beta\mu\nu}\del_\mu \phi\del_{\nu}\phi)
\del_\alpha Z^{I}\phi\del_\beta Z^{I}\phi dx\Big|
:=|L+M+N|.
\endaligned
\ee 

It is easy to get
\bel{6.16}
|L|\leq |C\delta s^{-1-\gamma}\int_{H_{s}}(s/t)^{3}\underline{\del}_\alpha Z^{I}\phi
\underline{\del}_\beta Z^{I}\phi dx|
\leq C\delta s^{-1-\gamma} E_{m}(s,\phi).
\ee
From \eqref{5.8}, we have
\bel{6.17}
|M|\leq C\epsilon^{2}s^{-2}E_{m}(s,\phi).
\ee
Then for $N$, we have
\begin{eqnarray}\label{6.18}
|N|&\leq& C\delta s^{-1-\gamma}|\int_{H_{s}}\sum_{|J|\leq1}(\frac{s}{t})^{3}\underline{\del}_\alpha \phi
\underline{\del}_\beta Z^{J}\phi
\underline{\del}_{\delta}Z^{I}\phi
\underline{\del}_{\gamma}Z^{I}\phi|
\leq
C\delta\epsilon^{2}s^{-3-\gamma}.
\end{eqnarray}

Combining \eqref{6.16}-\eqref{6.18}, we finally obtain the following. 

\begin{lemma}\label{Lemma6.3}
We have
\be
\Big|\int_{H_{s}}\frac{s}{t}\del_{t}G^{\alpha\beta}
\del_\alpha Z^{I}\phi\del_\beta Z^{I}\phi\Big|
\leq (C\delta s^{-1-\gamma}+C\epsilon^{2}s^{-3})E(s,Z^{I}\phi).
\ee
\end{lemma}

Based on above Lemmas \ref{Lemma6.1}-\ref{Lemma6.3}, we have the following proposition.

\begin{proposition}\label{pro6.4}
When $m\geq 6$, we have
\bel{6.20}
\frac{d}{ds}\Et_{m}(s,\phi)\leq (C\delta s^{-1-\gamma}+C\epsilon^{2}s^{-3})E_{m}(s,\phi).
\ee
\end{proposition}

\begin{proof}[Proof of Theorem \ref{Theorem1.4}]
The local existence and uniqueness of the solution can be derived by the standard method.
For the global existence, what we have to do is to ``close'' the a priori bootstrap estimate \eqref{5.1}.
By the standard bootstrap method and \eqref{6.20}, if we assume that $E_{m}(s,\phi)\leq C\epsilon^{2}$ for some large constant $C$, we deduce 
\be
E_m(s,\phi)\leq E_m(B+1,\phi)e^{\int_{B+1}^\infty C\delta s^{-1-\gamma}+C\epsilon^{2}s^{-3}ds}=C_{0}e^{C\delta+C\epsilon^{2}}\epsilon^{2}
\leq \frac{1}{2}C\epsilon^{2},
\ee
provided $\delta$ and $\epsilon$ are sufficiently small.
Thus, the solution exists globally in time and, furthermore its energy at all relevant order is uniformly bounded.
\end{proof}


\end{document}